\documentclass[11pt]{article}
\usepackage{amscd}
\usepackage{amsfonts}
\usepackage{amsmath}
\usepackage{amssymb}
\usepackage{amsthm}
\usepackage{bbm}
\usepackage{CJK}
\usepackage{fancyhdr}
\usepackage{graphicx}
\usepackage{indentfirst}
\usepackage{latexsym,bm}
\usepackage{mathrsfs}
\usepackage{xypic}
\usepackage[square, comma, sort&compress, numbers]{natbib}
\usepackage[colorlinks,linkcolor=red, anchorcolor=blue]{hyperref}
\allowdisplaybreaks[4]
\newtheorem{theorem}{Theorem}[section]
\newtheorem{lemma}[theorem]{Lemma}

\newtheorem{proposition}[theorem]{Proposition}

\newtheorem{corollary}[theorem]{Corollary}

\usepackage[top=1in,bottom=1in,left=1.25in,right=1.25in]{geometry}
\def\<{\langle}
\def\>{\rangle}

\def\o{\otimes}

\def\r{\rho}

\date{}
\begin{document}
\renewcommand{\baselinestretch}{1.2}
\renewcommand{\arraystretch}{1.0}
\title{\bf Yetter-Drinfeld-Long bimodules are modules}
 \date{}
\author {{\bf Daowei Lu\footnote {Corresponding author:  ludaowei620@126.com} \quad  Shuanhong Wang}\\
{\small Department of Mathematics, Southeast University}\\
{\small Nanjing, Jiangsu 210096, P. R. of China}}
 \maketitle
\begin{center}
\begin{minipage}{12.cm}

\noindent{\bf Abstract.} Let $H$ be a finite dimensional bialgebra. In this paper, we prove that the category of Yetter-Drinfeld-Long bimodules is isomorphic to the Yetter-Drinfeld category over the tensor product bialgebra $H\o H^*$ as monoidal category. Moreover if $H$ is a Hopf algebra with bijective antipode, the isomorphism is braided.
\\

\\

 \end{minipage}
 \end{center}
 \normalsize\vskip1cm

\section*{Introduction}

F. Panaite and F. V. Oystaeyen in \cite{PO} introduced the notion of $L$-$R$ smash biproduct, with the $L$-$R$ smash product and $L$-$R$ smash coproduct introduced in \cite{PV} as multiplication, respectively
comultiplication. When an object $A$ which is both an algebra and a coalgebra and a bialgebra $H$ form a $L$-$R$-admissible pair $(H,A)$, $A\natural H$ becomes a bialgebra with smash product and smash coproduct, and the Radford biproduct is a special case. It turns out that $A$ is in fact a bialgebra in the category $\mathcal{LR}(H)$ of Yetter-Drinfeld-Long bimodules (introduced in \cite{PO}) with some compatible condition.

The aim of this paper is to show that the category $\mathcal{LR}(H)$ coincides with the Yetter-Drinfeld category over the bialgebra $H\o H^*$, in the case when $H$ is finite dimensional. Hence any object $M\in\mathcal{LR}(H)$ is just a module over the Drinfeld double $D(H\o H^*)$.

The paper is organized as follows. In section 1, we recall the category $\mathcal{LR}(H)$. In section 2, we give the main result of this paper.

Throughout this article, all the vector spaces, tensor product and homomorphisms are over a fixed field $k$.  For a coalgebra $C$, we will use the Heyneman-Sweedler's notation $\Delta(c)=  c_{1}\otimes c_{2},$
for any $c\in C$ (summation omitted).

\section{Preliminaries}
\def\theequation{1.\arabic{equation}}
\setcounter{equation} {0}

Let $H$ be a bialgebra. The category $\mathcal{LR}(H)$ is defined as follows. The objects of $\mathcal{LR}(H)$ are vector spaces $M$ endowed with $H$-bimodule and $H$-bicomodule
structures (denoted by $h\o m \mapsto h\cdot m, m\o h \mapsto m\cdot h, m\mapsto m_{(-1)}\o m_{(0)},
m\mapsto m_{<0>}\o m_{<1>}$, for all $h\in H$, $m\in M$), such that $M$ is a left-left Yetter-Drinfeld module, a
left-right Long module, a right-right Yetter-Drinfeld module and a right-left Long module, i.e.
\begin{eqnarray}
&&(h_1\cdot m)_{(-1)}h_2\o(h_1\cdot m)_{(0)}=h_1m_{(-1)}\o h_2\cdot m_{(0)},\label{1a}\\
&&(h\cdot m)_{<0>}\o(h\cdot m)_{<1>}=h\cdot m_{<0>}\o m_{<1>},\label{1b}\\
&&(m\cdot h_2)_{<0>}\o h_1(m\cdot h_2)_{<1>}=m_{<0>}\cdot h_1\o m_{<1>} h_2,\label{1c}\\
&&(m \cdot h)_{(-1)}\o(m \cdot h)_{(0)}=m_{(-1)}\o m_{(0)}\cdot h.\label{1d}
\end{eqnarray}
The morphisms in $\mathcal{LR}(H)$ are $H$-bilinear and $H$-bicolinear maps.

If $H$ has a bijective antipode $S$, $\mathcal{LR}(H)$ becomes a strict braided monoidal category with the following structures: for all $M,N\in\mathcal{LR}(H)$, and $m\in M,n\in N,h\in H$,
\begin{eqnarray*}
&&h\cdot(m\o n)=h_1\cdot m\o h_2\cdot n,\quad (m\o n)_{(-1)}\o(m\o n)_{(0)}=m_{(-1)}n_{(-1)}\o m_{(0)}\o n_{(0)},\\
&&(m\o n)\cdot h=m\cdot h_1\o n\cdot h_2,\quad (m\o n)_{<0>}\o(m\o n)_{<1>}=m_{<0>}\o n_{<0>}\o m_{<1>}n_{<1>},
\end{eqnarray*}
and the braiding
$$c_{_{M,N}}:M\o N\mapsto N\o M,\quad m\o n\mapsto m_{(-1)}\cdot n_{<0>}\o m_{(0)}\cdot n_{<1>},$$
and the inverse
$$c^{-1}_{_{M,N}}:N\o M\mapsto M\o N,\quad n\o m\mapsto m_{(0)}\cdot S^{-1}(n_{<1>})\o S^{-1}(m_{(-1)})\cdot n_{<0>}.$$

\section{Main result}
\def\theequation{2.\arabic{equation}}
\setcounter{equation} {0} \hskip\parindent

In this section, we will give the main result of this paper.

\begin{lemma}
Let $H$ be a finite dimensional bialgebra. Then we have a functor $F:\mathcal{LR}(H)\longrightarrow\! ^{H\o H^*}_{H\o H^*}\mathcal{YD}$ given for any object $M\in\mathcal{LR}(H)$ and any morphism $\vartheta$ by
$$F(M)=M\quad\hbox{and}\quad F(\vartheta)=\vartheta,$$
where $H\o H^*$ is a bialgebra with tensor product and tensor coproduct.
\end{lemma}

\begin{proof}
For all $M\in\mathcal{LR}(H)$, first of all, define the left action of $H\o H^*$ on $M$ by
\begin{equation}
(h\o f)\cdot m=\langle f,m_{<1>}\rangle h\cdot m_{<0>},
\end{equation}
for all $h\in H,f\in H^*$ and $m\in M$. Then $M$ is a left $H\o H^*$-module. Indeed for all $h,h'\in H,f,f'\in H^*$ and $m\in M$,
\begin{align*}
(h\o f)(h'\o f')\cdot m&=(hh'\o ff')\cdot m\\
                      &=\langle ff',m_{<1>}\rangle hh'\cdot m_{<0>}\\
                      &=\langle f,m_{<1>1}\rangle\langle f',m_{<1>2}\rangle h\cdot(h'\cdot m_{<0>})\\
                      &=\langle f,m_{<0><1>}\rangle\langle f',m_{<1>}\rangle h\cdot(h'\cdot m_{<0><0>})\\
                      &\stackrel{(\ref{1b})}{=}\langle f,(h'\cdot m_{<0>})_{<1>}\rangle\langle f',m_{<1>}\rangle h\cdot(h'\cdot m_{<0>})_{<0>}\\
                      &=\langle f',m_{<1>}\rangle (h\o f)\cdot(h'\cdot m_{<0>})\\
                      &=(h\o f)\cdot((h'\o f')\cdot m).
\end{align*}
And
\begin{align*}
(1\o\varepsilon)\cdot m=\langle \varepsilon,m_{<1>}\rangle  m_{<0>}=m,
\end{align*}
as claimed. Next for all $m\in M$, define the left coaction of $H\o H^*$ on $M$ by
\begin{equation}
\r(m)=m_{[-1]}\o m_{[0]}=\sum m_{(-1)}\o h^i\o m_{(0)}\cdot h_i,
\end{equation}
where $\{h_i\}_i$ and $\{h^i\}_i$ are dual bases in $H$ and $H^*$. Then on one hand,
\begin{align*}
(\Delta_{H\o H^*}\o id)\r(m)=\sum m_{(-1)1}\o h^i_1\o m_{(-1)2}\o h^i_2\o m_{(0)}\cdot h_i.
\end{align*}
Evaluating the right side of the equation on $id\o g\o id\o h\o id$, we obtain
$$m_{(-1)1}\o m_{(-1)2}\o m_{(0)}\cdot gh.$$
On the other hand
\begin{align*}
(id\o\r)\r(m)&=\sum m_{(-1)}\o h^i\o(m_{(0)}\cdot h_i)_{(-1)}\o h^j\o (m_{(0)}\cdot h_i)_{(0)}\cdot h_j\\
             &\stackrel{(\ref{1d})}{=}\sum m_{(-1)}\o h^i\o m_{(0)(-1)}\o h^j\o (m_{(0)(0)}\cdot h_i)\cdot h_j\\
             &=\sum m_{(-1)1}\o h^i\o m_{(-1)2}\o h^j\o m_{(0)}\cdot h_ih_j.
\end{align*}
Evaluating the right side of the equation on $id\o g\o id\o h\o id$, we obtain
$$m_{(-1)1}\o m_{(-1)2}\o m_{(0)}\cdot gh.$$
Since $g,h\in H$ were arbitrary, we have
$$(\Delta_{H\o H^*}\o id)\r=(id\o\r)\r.$$
And since
$$(\varepsilon_{H\o H^*}\o id)(\r(m))=\varepsilon(m_{(-1)})m_{(0)}=m,$$
$M$ is a left $H\o H^*$-comodule.

Finally
\begin{align*}
&[(h\o f)_1\cdot m]_{[-1]}(h\o f)_2\o[(h\o f)_1\cdot m]_{[0]}\\
&=(h_1\cdot m_{<0>})_{[-1]}\langle f_1,m_{<1>}\rangle(h_2\o f_2)\o(h_1\cdot m_{<0>})_{[0]}\\
&=\sum\langle f_1,m_{<1>}\rangle((h_1\cdot m_{<0>})_{(-1)}h_2\o h^i f_2)\o(h_1\cdot m_{<0>})_{(0)}\cdot h_i\\
&\stackrel{(\ref{1a})}{=}\sum\langle f_1,m_{<1>}\rangle h_1m_{<0>(-1)}\o h^i f_2\o h_2\cdot m_{<0>(0)}\cdot h_i.
\end{align*}
Evaluating the right side of the equation on $id\o g\o id$, we obtain
$$\langle f,m_{<1>}g_2\rangle h_1m_{<0>(-1)}\o h_2\cdot m_{<0>(0)}\cdot g_1.$$
And
\begin{align*}
&(h\o f)_1m_{[-1]}\o(h\o f)_2\cdot m_{[0]}\\
&=\sum(h_1\o f_1)(m_{(-1)}\o h^i)\o(h_2\o f_2)\cdot(m_{(0)}\cdot h_i)\\
&=\sum h_1m_{(-1)}\o f_1h^i\o \langle f_2,(m_{(0)}\cdot h_i)_{<1>}\rangle h_2\cdot(m_{(0)}\cdot h_i)_{<0>}.
\end{align*}
Evaluating the right side of the equation on $id\o g\o id$, we obtain
\begin{align*}
&h_1m_{(-1)}\o \langle f,g_1(m_{(0)}\cdot g_2)_{<1>}\rangle h_2\cdot (m_{(0)}\cdot g_2)_{<0>}\\
&\stackrel{(\ref{1c})}{=}h_1m_{(-1)}\o \langle f,m_{(0)<1>}g_2\rangle h_2\cdot m_{(0)<0>}\cdot g_1\\
&=\langle f,m_{<1>}g_2\rangle h_1m_{<0>(-1)}\o h_2\cdot m_{<0>(0)}\cdot g_1.
\end{align*}
Therefore $M$ is a left-left Yetter-Drinfeld module over $H\o H^*$. It is straightforward to verify that any morphism in $\mathcal{LR}(H)$ is also a morphism in $^{H\o H^*}_{H\o H^*}\mathcal{YD}$. The proof is completed.
\end{proof}

\begin{lemma}
Let $H$ be a finite dimensional bialgebra. Then we have a functor $G:\ ^{H\o H^*}_{H\o H^*}\mathcal{YD}\longrightarrow \mathcal{LR}(H)$ given for any object $M\in\! ^{H\o H^*}_{H\o H^*}\mathcal{YD}$ and any morphism $\theta$ by
$$G(M)=M\quad\hbox{and}\quad G(\theta)=\theta.$$
\end{lemma}

\begin{proof}
We denote by $\varepsilon^*$ the map $\varepsilon_{H^*}$ defined by $\varepsilon_{H^*}(f)=f(1)$ for all $f\in H^*$.
For any $M\in\! ^{H\o H^*}_{H\o H^*}\mathcal{YD}$, denote the left $H\o H^*$-coaction on $M$ by
$$m\mapsto m_{[-1]}\o m_{[0]},$$
for all $m\in M.$ Define the $H$-bimodule and $H$-bicomodule structures as follows:
\begin{align}
&h\cdot m=(h\o\varepsilon)\cdot m,\quad \r_L(m)=m_{(-1)}\o m_{(0)}=(id\o\varepsilon^*)(m_{[-1]})\o m_{[0]},\\
&m\cdot h=\langle (\varepsilon\o id)m_{[-1]},h\rangle m_{[0]},\quad \r_R(m)=m_{<0>}\o m_{<1>}=\sum(1\o h^i)\cdot m\o h_i.
\end{align}
for all $h\in H$.

Obviously $M$ is a left $H$-module. And
\begin{align*}
(\Delta\o id)\r_L(m)&=\Delta((id\o\varepsilon^*)(m_{[-1]}))\o m_{[0]}\\
                    &=(id\o\varepsilon^*)(m_{[-1]1})(id\o\varepsilon^*)(m_{[-1]2})\o m_{[0]}\\
                    &=(id\o\varepsilon^*)(m_{[-1]})(id\o\varepsilon^*)(m_{[0][-1]})\o m_{[0][0]}\\
                    &=(id\o\r_L)\r_L(m).
\end{align*}
The counit is straightforward. Thus $M$ is a left $H$-comodule. For all $h,h'\in M$,
\begin{align*}
m\cdot hh'&=\langle (\varepsilon\o id)m_{[-1]},hh'\rangle m_{[0]}\\
          &=\langle (\varepsilon\o id)m_{[-1]1},h\rangle\langle (\varepsilon\o id)m_{[-1]2},h'\rangle m_{[0]}\\
          &=\langle (\varepsilon\o id)m_{[-1]},h\rangle\langle (\varepsilon\o id)m_{[0][-1]},h'\rangle m_{[0][0]}\\
          &=\langle (\varepsilon\o id)m_{[-1]},h\rangle m\cdot h'\\
          &=(m\cdot h)\cdot h'.
\end{align*}
The unit is obvious. Thus $M$ is a right $H$-module. Since
\begin{align*}
(id\o\Delta)\r_R(m)&=\sum(1\o h^i)\cdot m\o h_{i1}\o h_{i2}\\
                   &=\sum(1\o h^ih^j)\cdot m\o h^j\o h^i\\
                   &=(\r_R\o id)\r_R(m),
\end{align*}
it
follows that M is a right $H$-comodule.
Moreover
\begin{align*}
(h\cdot m)\cdot h'&=((h\o\varepsilon)\cdot m)\cdot h'\\
                  &=\langle(\varepsilon\o id)((h\o\varepsilon)\cdot m)_{[-1]},h'\rangle ((h\o\varepsilon)\cdot m)_{[0]}\\
                  &=\langle(\varepsilon\o id)[((h_1\o\varepsilon)\cdot m)_{[-1]}(h_2\o\varepsilon)],h'\rangle ((h_1\o\varepsilon)\cdot m)_{[0]}\\
                  &\stackrel{(\ref{1a})}{=}\langle(\varepsilon\o id)((h_1\o\varepsilon) m_{[-1]}),h'\rangle (h_2\o\varepsilon)\cdot m_{[0]}\\
                  &=\langle(\varepsilon\o id)m_{[-1]},h'\rangle(h\o\varepsilon)\cdot m_{[0]}\\
                  &=h\cdot(m\cdot h').
\end{align*}
Thus $M$ is an $H$-bimodule. And
\begin{align*}
(\r_L\o id)\r_R(m)&=\sum(id\o\varepsilon^*)((1\o h^i)\cdot m)_{[-1]}\o((1\o h^i)\cdot m)_{[0]}\o h_i\\
                  &=\sum(id\o\varepsilon^*)[((1\o h^i_1)\cdot m)_{[-1]}(1\o h^i_2)]\o((1\o h^i_1)\cdot m)_{[0]}\o h_i\\
                  &\stackrel{(\ref{1a})}{=}\sum(id\o\varepsilon^*)((1\o h^i_1)m_{[-1]})\o(1\o h^i_2)\cdot m_{[0]}\o h_i\\
                  &=(id\o\r_R)\r_L(m).
\end{align*}
Thus $M$ is an $H$-bicomodule.

We now prove (\ref{1a}). For all $h\in H, m\in M$,
\begin{align*}
&(h_1\cdot m)_{(-1)}h_2\o(h_1\cdot m)_{(0)}\\
&=((h_1\o\varepsilon)\cdot m)_{(-1)}h_2\o((h_1\o\varepsilon)\cdot m)_{(0)}\\
&=(id\o\varepsilon^*)(((h_1\o\varepsilon)\cdot m)_{[-1]}(h_2\o\varepsilon))\o((h_1\o\varepsilon)\cdot m)_{[0]}\\
&\stackrel{(\ref{1a})}{=}(id\o\varepsilon^*)((h_1\o\varepsilon) m_{[-1]})\o(h_2\o\varepsilon)\cdot m_{[0]}\\
&=h_1m_{(-1)}\o h_2\cdot m_{(0)}.
\end{align*}
We now prove (\ref{1b}):
\begin{align*}
(h\cdot m)_{<0>}\o(h\cdot m)_{<1>}&=((h\o\varepsilon)\cdot m)_{<0>}\o((h\o\varepsilon)\cdot m)_{<1>}\\
                                  &=\sum(1\o h^i)(h\o\varepsilon)\cdot m\o h_i\\
                                  &=\sum(h\o\varepsilon)(1\o h^i)\cdot m\o h_i\\
                                  &=h\cdot m_{<0>}\o m_{<1>}.
\end{align*}
We now prove (\ref{1c}): On one hand,
\begin{align*}
(m\cdot h_2)_{<0>}\o h_1(m\cdot h_2)_{<1>}&=\langle(\varepsilon\o id)m_{[-1]},h_2\rangle m_{[0]<0>}\o h_1 m_{[0]<1>}\\
                                          &=\sum\langle(\varepsilon\o id)m_{[-1]},h_2\rangle(1\o h^{i})\cdot m_{[0]}\o h_{1}h_{i}.
\end{align*}
Evaluating the right side on $id \o f$ for all $f\in H^*$, we have
\begin{align*}
 &\langle(\varepsilon\o id)m_{[-1]},h_2\rangle (1\o f_2)\cdot m_{[0]}f_1(h_1)\\
 &=\langle(\varepsilon\o id)(1\o f_1)m_{[-1]},h\rangle (1\o f_2)\cdot m_{[0]}\\
 &\stackrel{(\ref{1a})}{=}\langle(\varepsilon\o id)(((1\o f_1)\cdot m)_{[-1]}(1\o f_2)),h\rangle ((1\o f_1)\cdot m)_{[0]}.
\end{align*}
On the other hand,
\begin{align*}
m_{<0>}\cdot h_1\o m_{<1>}h_2&=\sum((1\o h^{i})\cdot m)\cdot h_1\o h_ih_2\\
&=\sum\langle(\varepsilon\o id)((1\o h^i)\cdot m)_{[-1]},h_1\rangle((1\o h^{i})\cdot m)_{[0]}\o h_{i}h_{2}
\end{align*}
Evaluating the right side on $id \o f$, we have
\begin{align*}
 &\langle(\varepsilon\o id)((1\o f_1)\cdot m)_{[-1]},h_1\rangle ((1\o f_1)\cdot m)_{[0]}f_2(h_2)\\
 &=\langle(\varepsilon\o id)(((1\o f_1)\cdot m)_{[-1]}(1\o f_2)),h\rangle ((1\o f_1)\cdot m)_{[0]}.
\end{align*}
Hence $(m\cdot h_2)_{<0>}\o h_1(m\cdot h_2)_{<1>}=m_{<0>}\cdot h_1\o m_{<1>}h_2$ since $f$ was arbitrary.

We now prove (\ref{1d}):
\begin{align*}
&(m\cdot h)_{(-1)}\o(m\cdot h)_{(0)}\\
&=\langle(\varepsilon\o id)m_{[-1]},h\rangle(id\o\varepsilon^*)(m_{[0][-1]})\o m_{[0][0]}\\
&=\langle(\varepsilon\o id)m_{[-1]1},h\rangle(id\o\varepsilon^*)(m_{[-1]2})\o m_{[0]}\\
&=(id\o h)m_{[-1]}\o m_{[0]}\\
&=\langle(\varepsilon\o id)m_{[-1]2},h\rangle(id\o\varepsilon^*)(m_{[-1]1})\o m_{[0]}\\
&=m_{(-1)}\o m_{(0)}\cdot h,
\end{align*}
where in the third equality, $(id\o h)m_{[-1]}$ means the second factor of $m_{[-1]}$ acts on $h$.

Therefore $M\in\mathcal{LR}(H)$. It is straightforward to verify that any morphism in $^{H\o H^*}_{H\o H^*}\mathcal{YD}$ is also a morphism in $\mathcal{LR}(H)$. The proof is completed.
\end{proof}

\begin{theorem}
Let $H$ be a finite dimensional bialgebra. Then we have a monoidal category isomorphism
$$\mathcal{LR}(H)\cong\!^{H\o H^*}_{H\o H^*}\mathcal{YD}.$$
Moreover if $H$ is a Hopf algebra with bijective antipode $S$, they are isomorphic as braided monoidal categories. Consequently
$$\mathcal{LR}(H)\cong\!_{D(H\o H^*)}\mathcal{M},$$
where $D(H\o H^*)$ is the Drinfeld double of $H\o H^*$.
\end{theorem}

\begin{proof}
It is easy to see that the functor $F:\mathcal{LR}(H)\longrightarrow\!^{H\o H^*}_{H\o H^*}\mathcal{YD}$ is monoidal and that $F\circ G=id$ and $G\circ F=id$. And for all $M,N\in\mathcal{LR}(H)$, and $m\in M,n\in N$,
\begin{align*}
m_{[-1]}\cdot n\o m_{[0]}&\stackrel{(2.2)}{=}\sum (m_{(-1)}\o h^i)\cdot n\o m_{(0)}\cdot h_i\\
                         &\stackrel{(2.1)}{=}\sum m_{(-1)}\cdot n_{<0>}\o m_{(0)}\cdot n_{<1>}.
\end{align*}
The proof is completed.
\end{proof}

\begin{corollary}
$(A,H)$ is an L-R-admissible pair if and only if $(A,H\o H^*)$ is an admissible pair.
\end{corollary}

By the isomorphism in Theorem 2.3, we can obtain the following result in \cite{PS} directly.

\begin{proposition}
Let $H$ be a finite dimensional Hopf algebra. The canonical braiding of $\mathcal{LR}(H)$ is pseudosymmetric if and only if $H$ is commutative and cocommutative.
\end{proposition}

\begin{proof}
From \cite{PSO}, the canonical braiding of $^{H\o H^*}_{H\o H^*}\mathcal{YD}$ is pseudosymmetric if and only if $H\o H^*$ is commutative and cocommutative. By the bialgebra structure of $H\o H^*$, the proof is completed.

\end{proof}

\section*{Acknowledgements}

This work was supported by the NSF of China (No. 11371088) and the Fundamental Research Funds for the Central Universities (No. KYLX15\_0109).

\end{document}